\documentclass{amsart}
\usepackage{amsmath}
\usepackage{amssymb}
\usepackage{amsthm}
\usepackage[top=3cm,left=3.5cm,right=3.5cm,bottom=3cm]{geometry} 
\usepackage{tgtermes}
\usepackage[T1]{fontenc}
\usepackage[shortlabels]{enumitem}
\usepackage[all]{xy}
\usepackage{color}
\usepackage{hyperref}
\numberwithin{equation}{section}
\theoremstyle{plain}
\newtheorem{theorem}{Theorem}[section]
\newtheorem{proposition}[theorem]{Proposition}
\newtheorem{lemma}[theorem]{Lemma}
\newtheorem{corollary}[theorem]{Corollary}

\theoremstyle{definition}
\newtheorem{definition}[theorem]{Definition}
\newtheorem{example}[theorem]{Example}
\theoremstyle{remark}
\newtheorem{remark}[theorem]{Remark}

\newcommand{\msf}[1]{\mathsf{#1}}
\newcommand{\mcal}[1]{\mathcal{#1}}
\newcommand{\mbb}[1]{\mathbb{#1}}
\newcommand{\mbf}[1]{\mathbf{#1}}
\newcommand{\mrm}[1]{\mathrm{#1}}
\newcommand{\mfk}[1]{\mathfrak{#1}}
\newcommand{\perf}{\mathrm{perf}}
\newcommand{\cyc}{\mathsf{cyc}}
\newcommand{\Rel}{\mathsf{Rel}}

\DeclareMathOperator{\Aut}{Aut}
\DeclareMathOperator{\hofib}{hofib}
\DeclareMathOperator{\cone}{cone}

\DeclareMathOperator{\CH}{CH}
\DeclareMathOperator{\Spec}{Spec}

\newdir{ >}{{}*!/-5pt/@{>}}

\title[Relative $K_0$ and relative cycle class map]{Relative $K_0$ and relative cycle class map}
\author{Ryomei Iwasa}
\address{Graduate School of Mathematical Sciences, the University of Tokyo, 3-8-1 Komaba, Meguro-ku, Tokyo, 153-8914 Japan.}
\email{ryomei@ms.u-tokyo.ac.jp}
\thanks{}

\begin{document}

\begin{abstract}
Let $F\colon \mcal{A}\to\mcal{B}$ be an exact functor between small exact categories.
We study the zeroth homotopy group $K_0(F)$ of the homotopy fiber of the map $K(\mcal{A})\to K(\mcal{B})$ between $K$-theory spectra.
Under the assumption that $F$ is a cofinal and that $\mcal{B}$ is split exact, we give an explicit description of $K_0(F)$ in terms of the triangulated functor $\msf{D}^b(\mcal{A})\to \msf{D}^b(\mcal{B})$ between the derived categories.

We apply it to the pair $(X,D)$ of a scheme $X$ and an affine closed subscheme $D$ of $X$, and get a description of the relative $K_0$-group $K_0(X,D)$ in terms of perfect complexes; it is generated by pairs of two perfect complexes of $X$ together with quasi-isomorphisms along $D$.
This description makes it possible to assign a cycle class in $K_0(X,D)$ to a cycle on $X$ not meeting $D$ in an intuitive way.
When $X$ is a separated regular scheme of finite type over a field and $D$ is an affine effective Cartier divisor on $X$, we prove that the cycle classes induce a surjective group homomorphism from the Chow group with modulus $\CH_*(X|D)$ defined by Binda-Saito to a suitable subquotient of $K_0(X,D)$.
\end{abstract}

\maketitle

\tableofcontents

\section{Introduction}

\subsection{}
Let $X$ be a separated regular scheme of finite type over a field.
Then, to every integral closed subscheme $V$ of $X$, we can assign the cycle class $\cyc(V)$ in the Grothendieck group $K_0(X)$ of algebraic vector bundles on $X$.
Grothendieck has shown that the cycle classes induce surjective group homomorphisms from the Chow groups to subquotients of $K_0(X)$
\begin{equation}\label{Gro-cyclemap}
	\cyc\colon \CH_k(X) \twoheadrightarrow F_kK_0(X)/F_{k-1}K_0(X)
\end{equation}
for all $k\ge 0$, cf.\ \cite[Exp 0, App.\ Ch II]{SGA6}.
Here, $F_*$ is the coniveau filtration; $F_kK_0(X)$ is generated by perfect complexes of $X$ whose supports are of dimension $\le k$. 

The current paper constructs a relative version of the cycle class map (\ref{Gro-cyclemap}).
Let $D$ be an effective Cartier divisor on $X$.
We are interested in the relative $K_0$-group $K_0(X,D)$, which is defined to be the zeroth homotopy group of the homotopy fiber of the canonical map $K(X)\to K(D)$ between $K$-theory spectra.
As a cycle theoretical invariant, we use the Chow group with modulus $\CH_*(X|D)$ defined by Binda-Saito \cite{BS17}: It is the group generated by cycles on $X$ which do not meet $D$ divided by a variant of rational equivalence (see \S\ref{Chowmodulus} for details).
Here is the main theorem, which generalizes (\ref{Gro-cyclemap}).
\begin{theorem}[Theorem \ref{thm:cycleclass}, Lemma \ref{lem:coniveau}]\label{mainthm}
Let $X$ be a separated regular scheme of finite type over a field and $D$ an affine effective Cartier divisor on $X$.
Then there exist surjective group homomorphisms
\begin{equation}\label{cyclemapmodulus}
	\cyc\colon \CH_k(X|D) \twoheadrightarrow F_kK_0(X,D)/F_{k-1}K_0(X,D)
\end{equation}
for all $k\ge 0$, where $F_*$ is the ``coniveau filtration'' (Definition \ref{def:coniveau}).
Furthermore, if $D$ has an affine open neighborhood in $X$, then $F_{\dim X}K_0(X,D)=K_0(X,D)$.
\end{theorem}

In \cite{BK18}, Binda-Krishna constructed a cycle class map for zero cycles with modulus, namely a map from $\CH_0(X|D)$ to $K_0(X,D)$, for modulus pairs $(X,D)$ with $X$ smooth quasi-projective over a perfect field.
They have also shown that the cycle class map is injective if $X$ is affine and the base filed is algebraically closed.
Also, Binda \cite{Bi18} constructed a cycle class map for higher zero cycles with modulus using a slightly different (stronger) modulus condition.

If $X$ is a smooth quasi-projective scheme over a field, then Grothendieck's Riemann-Roch type formula implies that the cycle class map (\ref{Gro-cyclemap}) is a rational isomorphism.
In a subsequent paper \cite{IK18}, we prove that the relative cycle map (\ref{cyclemapmodulus}) is a rational isomorphism, at least when $X$ is smooth affine.

\subsection{}
Let $X$ be a separated regular noetherian scheme and $V$ an integral closed subscheme of $X$.
Then the coherent sheaf $\mcal{O}_V$ is a perfect complex of $X$, i.e.\ quasi-isomorphic to a bounded complex $E_\bullet$ of algebraic vector bundles on $X$, and the cycle class is given by $\cyc(V):=\sum(-1)^i[E_i]\in K_0(X)$.
Here, it is more natural to consider the group $K^{\mrm{perf}}_0(X)$ generated by perfect complexes of $X$ with the relation $[P]=[P']+[P'']$ for each exact triangle $P'\to P\to P''\to P'[1]$.
It follows from \cite[Exp 1, 6.4]{SGA6} that the canonical map
\begin{equation}\label{K-Kperf}
	K_0(X) \xrightarrow{\simeq} K^{\mrm{perf}}_0(X)
\end{equation}
is an isomorphism.
Under this isomorphism, the cycle class $\cyc(V)$ is just the class of the perfect complex $\mcal{O}_V$ in $K^{\mrm{perf}}_0(X)$.

Now, suppose we are given an affine closed subscheme $D$ of $X$, and we denote the inclusion $D\hookrightarrow X$ by $\iota$.
Then the relative $K_0$-group $K_0(X,D)$ is generated by pairs $(E,E')$ of algebraic vector bundles on $X$ together with isomorphisms $E\vert_D\xrightarrow{\simeq}E'\vert_D$ along $D$ (Theorem \ref{thm:Heller}).
As in the absolute case, perfect complexes are more appropriate to construct cycle classes.
We define $K^{\mrm{perf}}_0(X,D)$ to be the group generated by pairs $(P,P')$ of perfect complexes of $X$ together with isomorphisms $L\iota^*P\xrightarrow{\simeq}L\iota^*P'$ in the derived category of $D$ with suitable relations (Definition \ref{def:K_0(F)tri}).
Then we show that the canonical map
\begin{equation}\label{keyisom}
	K_0(X,D) \xrightarrow{\simeq} K^{\mrm{perf}}_0(X,D)
\end{equation}
is an isomorphism (Theorem \ref{thm:relKscheme}).
This is a generalization of (\ref{K-Kperf}).

Let $V$ be an integral closed subscheme of $X$ which does not meet $D$.
Then $\mcal{O}_V$ is a perfect complex of $X$ and $L\iota^*\mcal{O}_V\simeq 0$.
Hence, the pair $(\mcal{O}_V,0)$ gives an element of $K^{\mrm{perf}}_0(X,D)$, which we denote by $\cyc(V)$.
When $X$ is of finite type over a field and $D$ is a Cartier divisor, we show that $\cyc$ kills the relations that define $\CH_*(X|D)$ and get Theorem \ref{mainthm}.

The hardest part of the above argument is the proof of the isomorphism (\ref{keyisom}).
This isomorphism holds more generally for a certain type of exact functors between small exact categories.
Actually, in large part of this paper, we discuss relative $K$-theory of exact categories and triangulated categories in general.

Here is a brief summary of the contents of this paper.
Let $F\colon \mcal{A}\to\mcal{B}$ be an exact functor between small exact categories.
We define $K_0(F)$ to be the group generated by pairs $(P,Q)$ of two objects in $\mcal{A}$ together with isomorphisms $F(P)\xrightarrow{\simeq}F(Q)$ in $\mcal{B}$ with suitable relations (Definition \ref{def:K_0(F)exact}).
Then, under the assumption that $\mcal{B}$ is split exact and that $F$ is cofinal, $K_0(F)$ is isomorphic to the zeroth homotopy group of the homotopy fiber of the map $K(\mcal{A})\to K(\mcal{B})$ between the $K$-theory spectra (Theorem \ref{thm:Heller}).
This essentially follows from Heller's result in \cite{He65} and we explain it in the first section \S\ref{relKexact}.
The second section \S\ref{relKtri} is the technical heart of this paper.
We define a group $K_0(T)$ for a triangulated functor $T$ between small triangulated categories (Definition \ref{def:K_0(F)tri}) as an analogue of the $K_0$ for an exact functor between small exact categories.
Then, for an exact functor $F\colon\mcal{A}\to\mcal{B}$ between small exact categories with $\mcal{B}$ being split exact, we prove that $K_0(F)$ is isomorphic to the $K_0$ of the triangulated functor $\msf{D}^b(\mcal{A})\to D^b(\mcal{B})$ between the derived categories (Theorem \ref{thm:comparison}).
This is the general assertion of the isomorphism (\ref{keyisom}).
Finally, in the third section \S\ref{relcycle}, we prove Theorem \ref{mainthm}.

\subsection*{Acknowledgement}
I am thankful to Amalendu Krishna for discussions about cycle class maps.
Parts of this paper was written when I was in Hausdorff Research Institute for Mathematics as a participant of the trimester program ``$K$-theory and Related Fields''.
I thank the institute and the organizers for their hospitality.
This work was supported by JSPS KAKENHI Grant Number 16J08843, and by the Program for Leading Graduate Schools, MEXT, Japan.

\section{Relative $K_0$ of exact categories}\label{relKexact}

Let us start from a general construction of a category, which is used throughout this paper.
\begin{definition}\label{def:Rel}
Let $F\colon\mcal{A}\to\mcal{B}$ be a functor of categories.
We define a category $\Rel(F)$:
\begin{itemize}
\item Objects are triples $(P,\alpha,Q)$ with $P,Q\in\mcal{A}$ and $\alpha\colon F(P)\xrightarrow{\simeq}F(Q)$ an isomorphism in $\mcal{B}$.
\item Morphisms from $(P,\alpha,Q)$ to $(P',\alpha',Q')$ are pairs $(f,g)$ of morphisms $f\colon P\to P'$ and $g\colon Q\to Q'$ in $\mcal{A}$ which make the diagram
\[
\xymatrix{
	F(P) \ar[r]^{F(f)} \ar[d]^\alpha & F(P')\ar[d]^{\alpha'} \\
	F(Q) \ar[r]^{F(g)} & F(Q')
}
\]
commutative.
\end{itemize}
\end{definition}

\subsection{Heller's theorem}

Let $F\colon\mcal{A}\to\mcal{B}$ be an exact functor between exact categories.
We call a sequence
\[
\xymatrix@1{
	(P',\alpha',Q') \ar[r]^-{(f,g)} & (P,\alpha,Q) \ar[r]^-{(f',g')} & (P'',\alpha'',Q'')
}
\]
in $\Rel(F)$ \textit{exact} if $P'\xrightarrow{f} P\xrightarrow{f'} P''$ and $Q'\xrightarrow{g} Q\xrightarrow{g'} Q''$ are exact sequences in $\mcal{A}$.
Under this definition, $\Rel(F)$ is an exact category.

\begin{definition}\label{def:K_0(F)exact}
Let $F\colon\mcal{A}\to\mcal{B}$ be an exact functor between small exact categories.
We define $K_0(F)$ to be the group with the generators $[X]$, one for each $X\in\Rel(F)$, and with the following relations:
\begin{enumerate}[(a)]
\item For each exact sequence $X'\rightarrowtail X\twoheadrightarrow X''$ in $\Rel(F)$,
\[
	[X] = [X'] + [X''].
\]
\item For each pair $((P,\alpha,Q),(Q,\beta,R))$ of objects in $\Rel(F)$, 
\[
	[(P,\alpha,Q)] + [(Q,\beta,R)] = [(P,\beta\alpha,R)].
\]
\end{enumerate}
\end{definition}

\begin{definition}\label{def:split}
An exact category is \textit{split exact} if every exact sequence is split exact.
\end{definition}

\begin{definition}\label{def:cofinal}
An additive functor $F\colon\mcal{A}\to\mcal{B}$ between additive category is \textit{cofinal} if for every $B\in\mcal{B}$ there exists $B'\in\mcal{B}$ and $A\in\mcal{A}$ such that $F(A)\simeq B\oplus B'$.
\end{definition}

For a small exact category $\mcal{A}$, we denote Quillen's $K$-theory spectrum by $K(\mcal{A})$.
Here is a reinterpretation of Heller's result in \cite{He65}.
\begin{theorem}\label{thm:Heller}
Let $F\colon\mcal{A}\to\mcal{B}$ be an exact functor between small exact categories.
Suppose that $\mcal{B}$ is split exact and that $F$ is cofinal.
Then there exists a natural isomorphism of groups
\[
	K_0(F) \simeq \pi_0\hofib(K(\mcal{A})\xrightarrow{F}K(\mcal{B})).
\]
\end{theorem}
We give a proof in \S\ref{pf:Heller}.

\subsection{Basic properties}

Here, we collect some basic properties of relative $K_0$-groups of exact categories (Definition \ref{def:K_0(F)exact}), whose proof is immediate from the definition. 
\begin{lemma}\label{lem:K_0(F)exact}
Let $F\colon\mcal{A}\to\mcal{B}$ be an exact functor between small exact categories.
Then:
\begin{enumerate}[label={\upshape(\roman*)}]
\item $[0]=0$ in $K_0(F)$. If $X,Y\in\Rel(F)$ are isomorphic, then $[X]=[Y]$ in $K_0(F)$.
\item If $\gamma\colon P\xrightarrow{\simeq}Q$ is an isomorphism in $\mcal{A}$, then $[(P,F(\gamma),Q)]=0$ in $K_0(F)$.
\item For every $(P,\alpha,Q)\in\Rel(F)$, $[(P,\alpha,Q)] + [(Q,\alpha^{-1},P)] = 0$ in $K_0(F)$.
\item Every element of $K_0(F)$ has the form $[X]$ for some $X\in\Rel(F)$.
\end{enumerate}
\end{lemma}

\begin{definition}\label{def:looparrow}
Let $F\colon\mcal{A}\to\mcal{B}$ be an exact functor between small exact categories.
Given two objects $(P,\alpha,Q),(P',\alpha',Q')$ in $\Rel(F)$, we write 
\[
	(P',\alpha',Q')\looparrowright(P,\alpha,Q)
\]
if there exist $N\in\mcal{A}$ and a commutator $\gamma$ in $\Aut(F(Q))$ which fit into an exact sequence
\[
\xymatrix@1{
	(P',\alpha',Q') \ar@{ >->}[r] & (P,\gamma\alpha,Q) \ar@{->>}[r] & (N,1,N).
}
\]
\end{definition}

\begin{lemma}\label{lem:looparrow}
Let $F\colon\mcal{A}\to\mcal{B}$ be an exact functor between small exact categories and $X,X'\in\Rel(F)$.
If $X'\looparrowright X$, then $[X']=[X]$ in $K_0(F)$.
\end{lemma}

\begin{remark}\label{rem:looparrow}
In fact, if $\mcal{B}$ is split exact, then the converse holds, i.e.\ all relations of $K_0(F)$ are generated by $\looparrowright$ (and $\looparrowleft$).
\end{remark}

\subsection{Elementary transformations}

Let $\mcal{C}$ be an additive category.
Let $P,Q\in\mcal{C}$.
Suppose that $P$ and $Q$ have the forms $P=P_1\oplus P_2$ and $Q=Q_1\oplus Q_2$.
Then a homomorphism from $P$ to $Q$ can be expressed by a matrix
\[
	\begin{pmatrix} a_{11} & a_{12} \\ a_{21} & a_{22}\end{pmatrix}
	\colon P_1\oplus P_2\to Q_1\oplus Q_2,
\]
where $a_{ij}$ is a morphism $P_j\to Q_i$ in $\mcal{C}$.

\begin{definition}\label{def:elementary}
An endomorphism $\alpha$ of $P\in\mcal{C}$ is an \textit{elementary transformation} if there exists an embedding $\mcal{C}\hookrightarrow\overline{\mcal{C}}$ of additive categories and $\alpha$ is isomorphic to an endomorphism of $P_1\oplus P_2$ of the form
\[
	\begin{pmatrix} 1 & a \\ 0 & 1\end{pmatrix}
\]
for some $P_1,P_2\in\overline{\mcal{C}}$ and $a\colon P_2\to P_1$.
We denote by $E(P)$ the subgroup of $\Aut(P)$ generated by elementary transformations.
\end{definition}

\begin{lemma}\label{lem:elementary}
Let $P\in\mcal{C}$ and $\alpha\in E(P)$.
Then 
\[
	\alpha\oplus 1\colon P\oplus P \to P\oplus P
\]
is a commutator of $\Aut(P\oplus P)$.
\end{lemma}
\begin{proof}
We may assume that $\alpha$ is an elementary transformation, i.e.\ $\exists\beta\colon P\xrightarrow{\simeq}P_1\oplus P_2$ and
\[
	\beta\alpha\beta^{-1}=\begin{pmatrix} 1 & a \\ 0 & 1\end{pmatrix}
\]
for some $a\colon P_2\to P_1$.
Then
\[
	\begin{pmatrix} 1 & a & 0 \\ 0 & 1 & 0  \\ 0 & 0 & 1 \end{pmatrix}
\]
is a commutator of $P_1\oplus P_2\oplus P_2$.
Indeed, this is equal to 
\[
	\left[\begin{pmatrix} 1 & 0 & a \\ 0 & 1 & 0 \\ 0 & 0 & 1 \end{pmatrix},\begin{pmatrix} 1 & 0 & 0 \\ 0 & 1 & 0 \\ 0 & 1 & 1 \end{pmatrix}\right]
\]
in $\Aut(P_1\oplus P_2\oplus P_2)$.
This implies that $\alpha\oplus 1\colon P\oplus P \to P\oplus P$ is a commutator of $\Aut(P\oplus P)$.
\end{proof}

\begin{corollary}\label{cor:elementary}
Let $F\colon\mcal{A}\to\mcal{B}$ be an exact functor between small exact categories.
Let $(P,\alpha,Q)\in\Rel(F)$ and $\gamma\in E(F(Q))$.
Then
\[
	[(P,\alpha,Q)]=[(P,\gamma\alpha,Q)]
\]
in $K_0(F)$.
\end{corollary}
\begin{proof}
According to Lemma \ref{lem:elementary},
\[
	(P,\alpha,Q)\looparrowright (P\oplus Q,\alpha\oplus 1,Q\oplus Q) \looparrowleft (P,\gamma\alpha,Q).
\]
Hence, $[(P,\alpha,Q)]=[(P,\gamma\alpha,Q)]$ by Lemma \ref{lem:looparrow}.
\end{proof}

\subsection{Proof of Theorem \ref{thm:Heller}}\label{pf:Heller}

Let $\mcal{A}$ be a small exact category.
In \cite{Ne98}, Nenashev provides generators and relations for $K_1(\mcal{A})$; the generators are double exact sequences.
In case $\mcal{A}$ is split exact (Definition \ref{def:split}), the generators and relations become simpler and the resulting group coincides with the one considered by Heller in \cite{He65}.
\begin{lemma}
Let $\mcal{A}$ be a small exact category.
We define $K_1^{\mrm{He}}(\mcal{A})$ to be the abelian group with the generators $[h]$, one for each $P\in\mcal{A}$ and each $h\in\Aut(P)$, and with the following relations:
\begin{enumerate}[label={\upshape(\alph*)}]
\item For a commutative diagram
\[
\xymatrix{
	P' \ar@{>->}[r]^f \ar[d]^{h'}_\simeq & P \ar[d]^h_\simeq \ar@{->>}[r]^g & P'' \ar[d]^{h''}_\simeq \\
	P' \ar@{>->}[r]^f 					 & P \ar@{->>}[r]^g 				& P''
}
\]
with exact rows, $[h]=[h']+[h'']$.
\item $[h_2\circ h_1]=[h_2]+[h_1]$ for $h_1,h_2\in\Aut(P)$.
\end{enumerate}
If $\mcal{A}$ is split exact, then there exists a natural isomorphism
\[
	K_1^{\mrm{He}}(\mcal{A}) \xrightarrow{\simeq} K_1(\mcal{A}).
\]
\end{lemma}

Let $F\colon\mcal{A}\to\mcal{B}$ be an exact functor between small exact categories.
We suppose that $F$ is cofinal (Definition \ref{def:split}).
Then every element of $K_1^{\mrm{He}}(\mcal{B})$ is represented by $g\in\Aut(F(P))$ for some $P\in\mcal{A}$.
According to \cite[Proposition 4.2]{He65}, the class $[(P,g,P)]$ in $K_0(F)$ does not depend on the representative and gives a group homomorphism
\[
	\delta\colon K_1^{\mrm{He}}(\mcal{B}) \to K_0(F).
\]

We define a map $\iota\colon K_0(F)\to K_0(\mcal{A})$ by sending $[(P,\alpha,Q)]$ to $[P]-[Q]$.

\begin{proposition}[Heller]\label{prop:relK}
Let $F\colon\mcal{A}\to\mcal{B}$ be an exact functor between small exact categories.
Suppose that $\mcal{B}$ is split exact and that $F$ is cofinal.
Then the sequence
\[
\xymatrix@1{
	K_1(\mcal{A}) \ar[r]^F & K_1(\mcal{B}) \ar[r]^\delta & K_0(F) \ar[r]^\iota & K_0(\mcal{A}) \ar[r]^F & K_0(\mcal{B}).
}
\]
is exact.
\end{proposition}
\begin{proof}
Heller showed the sequence
\[
\xymatrix@1{
	K_1^{\mrm{He}}(\mcal{A}) \ar[r]^F & K_1^{\mrm{He}}(\mcal{B}) \ar[r]^\delta & K_0(F) \ar[r]^\iota & K_0(\mcal{A}) \ar[r]^F & K_0(\mcal{B}).
}
\]
is exact \cite[Proposition 5.2]{He65}.
Now, $K_1^{\mrm{He}}(\mcal{B})=K_1(\mcal{B})$ and the map $K_1^{\mrm{He}}(\mcal{A}) \to K_1^{\mrm{He}}(\mcal{B})$ factors through $K_1(\mcal{A})$.
This proves the exactness at $K_1(\mcal{B})$.
\end{proof}

\begin{proof}[Proof of Theorem \ref{thm:Heller}]
In \cite{GG87}, Gillet and Grayson have constructed a simplicial set $G\mcal{A}$ such that its geometric realization is naturally homotopy equivalent to $K(\mcal{A})$.
The $0$-simplexes of $G\mcal{A}$ are pairs $(P,Q)$ of objects in $\mcal{A}$.
The $1$-simplexes of $G\mcal{A}$ are pairs of exact sequences of the forms
\[
	\xymatrix@1{ P_0 \ar@{ >->}[r] & P_1 \ar@{->>}[r] & P_{01}}, \quad
	\xymatrix@1{ Q_0 \ar@{ >->}[r] & Q_1 \ar@{->>}[r] & P_{01}}.
\]
The face maps $G\mcal{A}_1\to G\mcal{A}_0$ send the above to $(P_0,Q_0)$ and $(P_1,Q_1)$ respectively.

Let $GF_0$ be the set of objects in $\Rel(F)$.
Let $GF_1$ be the set of all pairs $(l,\gamma)$ where $l$ is an exact sequence in $\Rel(F)$ of the form
\[
	l\colon\quad\xymatrix@1{(P,\alpha,Q)\ar@{ >->}[r] & (R,\beta,S)\ar@{->>}[r] & (N,1,N)}
\]
and $\gamma$ is a commutator of $\Aut(F(S))$.
We define face maps $d_1,d_2\colon GF_1\to GF_0$ by $d_1((l,\gamma)):=(P,\alpha,Q)$ and $d_2((l,\gamma)):=(R,\gamma\beta,S)$,
and a degeneracy map $s\colon GF_0\to GF_1$ by $s(X) := (X\xrightarrow{1}X\to 0,1)$.
We define $GF$ to be the simplicial set generated by $GF_1,GF_0$.
Then
\[
	\pi_0|GF|=K_0(F).
\]

We have a natural map $GF\to G\mcal{A}$ which sends $(P,\alpha,Q)\in GF_0$ to $(P,Q)$ and $(l,\gamma)\in GF_1$ to the underlying pair of exact sequences of $l$.
Then the composite $GF\to G\mcal{A}\to G\mcal{B}$ is homotopic to zero.
Therefore, we obtain a natural map
\[
	\theta\colon K_0(F)\to \pi_0\hofib(K(\mcal{A})\to K(\mcal{B})).
\]

According to Proposition \ref{prop:relK}, it remains to show that $\delta\colon K_1(\mcal{B})\to K_0(F)$ followed by $\theta$ is equal to the boundary map
\[
	\partial\colon K_1(\mcal{B})\to \pi_0\hofib(K(\mcal{A})\to K(\mcal{B})),
\]
and it is straightforward.
\end{proof}

\begin{remark}
The construction of $GF_0$ and $GF_1$ in the proof suggests that there might be an algebraic construction of $GF_2,GF_3,\dotsc$ so that the resulting simplicial set is a model of the double loop space of the relative Waldhausen construction $wS_\bullet(S_\bullet F)$ in, say, \cite[8.5.3]{We13}.
\end{remark}

\begin{example}\label{mainexample}
Here are examples of exact functors $F\colon\mcal{A}\to\mcal{B}$ between exact categories such that $\mcal{B}$ is split exact and that $F$ is cofinal.
\begin{enumerate}[(1)]
\item A base change functor $\mbf{P}(A)\to\mbf{P}(B)$ induced from a ring homomorphism $A\to B$.
Here, $\mbf{P}(-)$ is the category of finitely generated projective modules.
\item A base change functor $\msf{Vec}(X)\to\msf{Vec}(Y)$ induced from a morphism of schemes $Y\to X$ with $Y$ affine.
Here, $\msf{Vec}(-)$ is the category of algebraic vector bundles.
\end{enumerate}
\end{example}

\section{Relative $K_0$ of triangulated categories}\label{relKtri}

The goal in this section is to prove Theorem \ref{thm:comparison}.

\subsection{The definition and basic properties}

Let $F\colon\mcal{A}\to\mcal{B}$ be a triangulated functor between triangulated categories.
Refer to Definition \ref{def:Rel} for the definition of the category $\Rel(F)$.
We denote by $[1]$ the shift functor of $\mcal{A}$ or $\mcal{B}$, and we define an endofunctor of $\Rel(F)$ by $(P,\alpha,Q)[1]:=(P[1],\alpha[1],Q[1])$.
We call a sequence
\[
\xymatrix@1{
	(P_1,\alpha_1,Q_1)\ar[r]^-{(f_1,g_1)} & (P_2,\alpha_2,Q_2)\ar[r]^-{(f_2,g_2)} & (P_3,\alpha_3,Q_3)\ar[r]^-{(f_3,g_3)} & (P_1,\alpha_1,Q_1)[1]
}
\]
in $\Rel(F)$ an \textit{exact triangle} if
\begin{gather*}
\xymatrix@1{
	P_1\ar[r]^-{f_1} & P_2\ar[r]^-{f_2} & P_3\ar[r]^-{f_3} & P_1[1]
}\\
\xymatrix@1{
	Q_1\ar[r]^-{g_1} & Q_2\ar[r]^-{g_2} & Q_3\ar[r]^-{g_3} & Q_1[1]
}
\end{gather*}
are exact triangles in $\mcal{A}$.
Under this definition, $\Rel(F)$ is an additive category which satisfies the first two axioms (TR1) and (TR2) of triangulated category in \cite[1.1]{Ve77}, but may not satisfy the other axioms (TR3) nor (TR4).

\begin{definition}\label{def:K_0(F)tri}
Let $F\colon\mcal{A}\to\mcal{B}$ be a triangulated functor between small triangulated categories.
We define $K_0(F)$ to be the group with the generators $[X]$, one for each $X\in\Rel(F)$, and with the following relations:
\begin{enumerate}[(a)]
\item For each exact triangle $X'\to X\to X''\to X'[1]$ in $\Rel(F)$,
\[
	[X] = [X'] + [X''].
\]
\item For each pair $((P,\alpha,Q),(Q,\beta,R))$ of objects in $\Rel(F)$, 
\[
	[(P,\alpha,Q)] + [(Q,\beta,R)] = [(P,\beta\alpha,R)].
\]
\end{enumerate}
\end{definition}

The same properties as in Lemma \ref{lem:K_0(F)exact} also hold for triangulated categories.
Again, the proof is immediate.

\begin{lemma}\label{lem:K_0(F)tri}
Let $F\colon\mcal{A}\to\mcal{B}$ be a triangulated functor between small triangulated categories.
Then:
\begin{enumerate}[label={\upshape(\roman*)}]
\item $[0]=0$ in $K_0(F)$. If $X,Y\in\Rel(F)$ are isomorphic, then $[X]=[Y]$ in $K_0(F)$.
\item For every $X\in\Rel(F)$ and every integer $n$, $[X[n]]=(-1)^n[X]$ in $K_0(F)$.
\item If $\gamma\colon P\xrightarrow{\simeq}Q$ is an isomorphism in $\mcal{A}$, then $[(P,F(\gamma),Q)]=0$ in $K_0(F)$.
\item For every $(P,\alpha,Q)\in\Rel(F)$, $[(P,\alpha,Q)] + [(Q,\alpha^{-1},P)] = 0$ in $K_0(F)$.
\item Every element of $K_0(F)$ has the form $[X]$ for some $X\in\Rel(F)$.
\end{enumerate}
\end{lemma}

\begin{lemma}
Let $F\colon\mcal{A}\to\mcal{B}$ be a triangulated functor between small triangulated categories.
Let $\mcal{A}_0$ be a thick triangulated subcategory of $\mcal{A}$ such that $F(\mcal{A}_0)=0$.
Then $F$ factors through the Verdier quotient $\mcal{A}/\mcal{A}_0$, and there is an exact sequence of abelian groups
\[
\xymatrix@1{
	K_0(\mcal{A}_0) \ar[r] & K_0(\mcal{A}\xrightarrow{F}\mcal{B}) \ar[r] & K_0(\mcal{A}/\mcal{A}_0\xrightarrow{\bar{F}}\mcal{B})\ar[r] & 0.
}
\]
\end{lemma}
\begin{proof}
Note that $\msf{Ob}(\Rel(F))=\msf{Ob}(\Rel(\bar{F}))$.
An exact triangle in $\Rel(\bar{F})$ is a sequence
\[
\xymatrix@1{
	(P_1,\alpha_1,Q_1)\ar[r]^-{(f_1,g_1)} & (P_2,\alpha_2,Q_2)\ar[r]^-{(f_2,g_2)} & (P_3,\alpha_3,Q_3)\ar[r]^-{(f_3,g_3)} & (P_1,\alpha_1,Q_1)[1]
}
\]
in $\Rel(\bar{F})$ such that 
\begin{gather*}
\xymatrix@1{
	P_1\ar[r]^-{f_1} & P_2\ar[r]^-{f_2} & P_3\ar[r]^-{f_3} & P_1[1]
}\\
\xymatrix@1{
	Q_1\ar[r]^-{g_1} & Q_2\ar[r]^-{g_2} & Q_3\ar[r]^-{g_3} & Q_1[1]
}
\end{gather*}
are isomorphic in $\mcal{A}/\mcal{A}_0$ to exact triangles in $\mcal{A}$.
It follows that $K_0(\bar{F})$ is the group with the generators $[X]$, one for each $X\in\Rel(F)$, and with the relations (a) (b) of Definition \ref{def:K_0(F)tri} and an additional relation
\begin{enumerate}
\item[(c)] For $X=(P,0,Q)$ with $P,Q\in\mcal{A}_0$, $[X]=0$.
\end{enumerate}
This says that $K_0(\bar{F})$ is the quotient of $K_0(F)$ by the image of $K_0(\mcal{A}_0)$.
\end{proof}

\subsection{Comparison theorem}\label{comparison}

For an additive category $\mcal{A}$, we use the following notation:
\begin{enumerate}[(1)]
\item $\msf{Ch}^b(\mcal{A})$ is the category of bounded chain complexes in $\mcal{A}$.
\item $\msf{K}^b(\mcal{A})$ is the bounded homotopy category, i.e.\ the same object with $\msf{Ch}^b(\mcal{A})$ and morphisms up to homotopy.
We regard $\msf{K}^b(\mcal{A})$ as a triangulated category in the standard way (cf.\ \cite{Ve77}).
\end{enumerate}

Here is the main theorem in this section.
\begin{theorem}\label{thm:comparison}
Let $\mcal{A}$ be a small exact category which is closed under the kernels of surjections.
Let $\mcal{B}$ be a small split exact category and $F\colon\mcal{A}\to\mcal{B}$ an exact functor.
We define $\msf{K}^{b,\emptyset}(\mcal{A})$ to be the full subcategory of $\msf{K}^b(\mcal{A})$ consisting acyclic complexes.
Then $F$ induces a triangulated functor
\[
	\msf{D}(F)\colon \msf{K}^b(\mcal{A})/\msf{K}^{b,\emptyset}(\mcal{A})\to \msf{K}^b(\mcal{B})
\]
and the canonical map
\[
	K_0(F) \xrightarrow{\simeq} K_0(\msf{D}(F))
\]
is an isomorphism.
\end{theorem}

Here, we have chosen an embedding of $\mcal{A}$ into an abelian category (such an embedding exists by Freyd-Mitchell theorem).
The terms ``kernels of surjections'' and ``acyclic complexes'' are understood in this abelian category.
The theorem implies that $K_0(\msf{D}(F))$ does not depend on the choice of the embedding.

The difficulty in the proof of Theorem \ref{thm:comparison} is how to define the inverse.
In absolute case, i.e.\ $\mcal{B}=0$, the inverse is clear, which is given by $E_\bullet\mapsto\sum(-1)^i[E_i]$.
We cannot imitate this directly because we have to keep track of homotopy equivalences in $\mcal{B}$.
However, a variant does still work.
We call it the Euler characteristic and study in \S\ref{EulerI} and \S\ref{EulerII}.
Using this machinery, the proof of Theorem \ref{thm:comparison} is completed in \S\ref{proof:comparison}.

We fix $\mcal{A}$, $\mcal{B}$ and $F$ as in Theorem \ref{thm:comparison} until the end of this section.

\subsection{Euler characteristic I}\label{EulerI}

Let $C$ be a bounded complex in $\mcal{B}$ which is homotopic to zero.
Then there exists $s\colon C_n\to C_{n+1}$ such that $sd+ds=1$ and $ss=0$, which we call a \textit{strict split} of $C$.
Then the direct sum of the maps
\[
\xymatrix@R-1pc{
							& C_{n+1} \\
	C_n \ar[ru]^s \ar[rd]_d & \\
							& C_{n-1},
}
\]
gives an isomorphism
\[
	\Phi_{C,s}\colon \bigoplus_n C_{2n+1} \xrightarrow{\simeq} \bigoplus_n C_{2n}.
\]
\begin{lemma}\label{lem:strictsplit}
If $s'$ is another strict split of $C$, then there exists an elementary transformation (Definition \ref{def:elementary}) $\gamma$ of $C_{2*}$ such that $\Phi_{C,s}=\gamma\Phi_{C,s'}$.
\end{lemma}
\begin{proof}
Take an embedding of $\mcal{B}$ into an abelian category $\overline{\mcal{B}}$ by Freyd-Mitchell theorem.
Let $Z_n\in \overline{\mcal{B}}$ be the kernel of $d\colon C_n\to C_{n-1}$.
Then $C$ decomposes into short exact sequences
\[
\xymatrix@1{
	0 \ar[r] & Z_n \ar[r]^-\epsilon & C_n \ar[r]^-\delta & Z_{n-1} \ar[r] & 0,
}
\]
and $s\epsilon\colon Z_{n-1}\to C_n$ or $\delta s\colon C_n\to Z_n$ give splits of these short exact sequences.
Hence, the map
\[
	\phi_{s,n}:=(\delta s,\delta)\colon C_n\xrightarrow{\simeq} Z_n\oplus Z_{n-1},
\]
is an isomorphism with the inverse $\phi_{s,n}^{-1}=(\epsilon,s\epsilon)$.

Now, it is clear that there is an elementary transformation $\gamma_n$ of $Z_n\oplus Z_{n-1}$ such that $\phi_{s,n}=\gamma_n\phi_{s',n}$.

On the other hand, $\phi_{s,n}$'s give isomorphisms
\[
	C_{2*+1} \xrightarrow[\simeq]{\phi_{s,2*+1}} Z_* \xrightarrow[\simeq]{\phi_{s,2*}^{-1}} C_{2*},
\]
and the composite is equal to $\Phi_{C,s}$ because $\epsilon\delta=d$ and $s\epsilon\delta s=sds=s$.
Therefore, there exists an elementary transformation $\gamma$ such that $\Phi_{C,s}=\gamma\Phi_{C',s}$.
\end{proof}

Let $P,Q$ be bounded complexes in $\mcal{A}$ and $\alpha$ a homotopy equivalence $F(P)\xrightarrow{\sim}F(Q)$.
We apply the above construction to $\cone\alpha$.
Set $\Phi:=\Phi_{\cone\alpha,s}$ for some strict split $s$ of $\cone\alpha$.
Thanks to Lemma \ref{lem:strictsplit} and Corollary \ref{cor:elementary}, the following is well-defined.

\begin{definition}\label{def:chi}
We define the \textit{Euler characteristic of $(P,\alpha,Q)$} by
\[
	\chi(P,\alpha,Q) := \Bigl[\Bigl(\bigoplus_n(P_{2n}\oplus Q_{2n+1}),\Phi,\bigoplus_n(P_{2n-1}\oplus Q_{2n})\Bigr)\Bigr]\in K_0(F).
\]
\end{definition}

Here are first properties of the Euler characteristic.
\begin{lemma}\label{lem:chi-first}
Let $P,Q$ be bounded complexes in $\mcal{A}$ with a homotopy equivalence $\alpha\colon F(P)\xrightarrow{\sim}F(Q)$.
\begin{enumerate}[label={\upshape(\roman*)}]
\item Let $P',Q'\in\msf{Ch}^b(\mcal{A})$ with isomorphisms of complexes $\gamma\colon F(P)\xrightarrow{\simeq}F(P')$ and $\delta\colon F(Q)\xrightarrow{\simeq}F(Q')$.
If $[(P_n,\gamma_n,P_n')]=[(Q_n,\delta_n,Q'_n)]=0$ for all $n$, then
\[
	\chi(P',\delta\alpha\gamma^{-1},Q')=\chi(P,\alpha,Q).
\]
\item For every integer $n$, $\chi(P[n],\alpha[n],Q[n])=(-1)^n\chi(P,\alpha,Q)$.
\end{enumerate}
\end{lemma}
\begin{proof}
(i) $\gamma:=(\gamma,\delta)$ gives an isomorphism of complexes $\cone\alpha\xrightarrow{\simeq}\cone(\delta\alpha\gamma^{-1})$.
Hence, $\Phi_{\cone(\delta\alpha\gamma^{-1})}$ is equal to the composite
\begin{multline*}
	F(P'_{2*})\oplus F(Q'_{2*+1}) \xrightarrow{(\gamma_{2*}^{-1},\delta_{2*+1}^{-1})} F(P_{2*})\oplus F(Q_{2*+1}) \\
		\xrightarrow{\Phi_{\cone\alpha}} F(P_{2*-1})\oplus F(Q_{2*}) \xrightarrow{(\gamma_{2*-1},\delta_{2*})} F(P'_{2*-1})\oplus F(Q'_{2*}).
\end{multline*}
It follows from the assumption that
\[
	(P'_{2*}\oplus Q'_{2*+1},(\gamma_{2*}^{-1},\delta_{2*+1}^{-1}),P_{2*}\oplus Q_{2*+1})
	= (P_{2*-1}\oplus Q_{2*},(\gamma_{2*-1},\delta_{2*}),P'_{2*-1}\oplus Q'_{2*}) = 0,
\]
and thus $\chi(P',\delta\alpha \gamma^{-1},Q') = \chi(P,\alpha,Q)$.

(ii) By the construction,
\[
	\chi(P[1],\alpha[1],Q[1]) = [(P_{2*}\oplus Q_{2*+1},-\Phi^{-1}_{\cone(-\alpha)},P_{2*-1}\oplus Q_{2*})].
\]
By Lemma \ref{lem:K_0(F)tri} (iv), the right hand side equals to
\[
	-[(P_{2*-1}\oplus Q_{2*},\Phi_{\cone(-\alpha)},P_{2*}\oplus Q_{2*+1})] = -\chi(P,-\alpha,Q),
\]
which equals to $-\chi(P,\alpha,Q)$ by (i).
\end{proof}

We will use the following lemma on homological algebra.
\begin{lemma}\label{lem:halg}
Let $\mcal{C}$ be an additive category.
Suppose given a commutative diagram
\[
\xymatrix{
	A' \ar@{ >->}[r]^{f_1} \ar@{->>}[d]^{d'} & A \ar@{->>}[r]^{g_1} \ar@{->>}[d]^d & A'' \ar@{->>}[d]^{d''} \\
	B' \ar@{ >->}[r]^{f_2} 			   		 & B \ar@{->>}[r]^{g_2} 			   & B''
}
\]
in $\mcal{C}$ such that the rows are split exact sequences and that $d',d''$ are split epimorphisms.
Let $s',s''$ be splits of $d',d''$.
Then there exists $\tilde{s}\colon B\to A$ such that $\tilde{s}f_2=f_1s'$, $g_1\tilde{s}=s''g_2$ and $\gamma:=d\tilde{s}$ is an elementary transformation of $B$.
In particular, $s:=\tilde{s}\gamma^{-1}$ is a split of $d$.
\end{lemma}
\begin{proof}
Let $a$ and $b$ are splits of $g_1$ and $f_2$ respectively;
\[
\xymatrix{
	A' \ar@{ >->}[r]_{f_1} \ar@{->>}[d] & A \ar@{->>}[r] \ar@{->>}[d]^d 	   & A'' \ar@{->>}[d] \ar@/_1pc/[l]_a  \\
	B' \ar@{ >->}[r] \ar@/^1pc/[u]^{s'} & B \ar@{->>}[r]^{g_2} \ar@/^1pc/[l]^b & B'' \ar@/_1pc/[u]_{s''}.
}
\]
We define
\[
	\tilde{s}:=f_1s'b+as''g_2\colon B\to A.
\]
Then $\tilde{s}f_2=f_1s'$ and $g_1\tilde{s}=s''g_2$.
Set $\gamma:=d\tilde{s}$.
Then $\gamma f_2=f_2$ and $g_2\gamma=g_2$, which implies that $\gamma$ is an elementary transformation of $B$.
\end{proof}

\subsection{Euler characteristic II}\label{EulerII}

In this subsection, we prove that the Euler characteristic defined in Definition \ref{def:chi} gives a group homomorphism $\chi\colon K_0(\msf{D}(F))\to K_0(F)$.

\begin{lemma}\label{lem:chi-exact}
Suppose we are given exact sequences
\[
	\xymatrix@1{P' \ar@{ >->}[r] & P \ar@{->>}[r] & P''}, \quad \xymatrix@1{Q' \ar@{ >->}[r] & Q \ar@{->>}[r] & Q''}
\]
in $\msf{Ch}^b(\mcal{A})$ and homotopy equivalences $\alpha,\alpha',\alpha''$ fitting into a commutative diagram
\[
\xymatrix{
	F(P') \ar[d]^{\alpha'}_\sim \ar@{ >->}[r] & F(P) \ar[d]^\alpha_\sim \ar@{->>}[r] & F(P'') \ar[d]^{\alpha''}_\sim \\
	F(Q') \ar@{ >->}[r] 					  & F(Q) \ar@{->>}[r] 					 & F(Q'').
}
\]
Then
\[
	\chi(P,\alpha,Q) = \chi(P',\alpha',Q') + \chi(P'',\alpha'',Q'').
\]
\end{lemma}
\begin{proof}
We have a commutative diagram
\[
\xymatrix{
	Z_n' \ar@{ >->}[r] \ar@{ >->}[d]^{\epsilon'} & Z_n \ar@{->>}[r] \ar@{ >->}[d]^\epsilon & Z_n'' \ar@{ >->}[d]^{\epsilon''} \\
	F(P_{n-1}')\oplus F(Q_n') \ar@{ >->}[r] \ar@{->>}[d]^{\delta'} & F(P_{n-1})\oplus F(Q_n) \ar@{->>}[r] \ar@{->>}[d]^\delta 
		& F(P_{n-1}'')\oplus F(Q_n'') \ar@{->>}[d]^{\delta''} \\
	Z_{n-1}' \ar@{ >->}[r] & Z_{n-1} \ar@{->>}[r] & Z_{n-1}''
}
\]
with exact rows and columns, where $Z_*,Z_*',Z_*''$ are the kernels of the differentials of $\cone\alpha,\cone\alpha',\cone\alpha''$.
We take splits $s'$ and $s''$ of $\delta'$ and $\delta''$.
Let $\tilde{s}$ be a map $Z_{n-1}\to F(P_{n-1})\oplus F(Q_n)$ as in Lemma \ref{lem:halg}, and set $\gamma:=\delta\tilde{s}$, $s:=\tilde{s}\gamma$.
Then $\gamma$ is an elementary transformation and $s$ is a split of $\delta$.

Now, we have isomorphisms $\tilde{\phi}_n,\phi_n',\phi_n''$ fitting into a commutative diagram
\[
\xymatrix{
	Z_n'\oplus Z_{n-1}' \ar[d]^{\phi_n'=(\epsilon',s')}_\simeq \ar@{ >->}[r] & Z_n\oplus Z_{n-1} \ar[d]^{\tilde{\phi}_n=(\epsilon,\tilde{s})}_\simeq \ar@{->>}[r]
		& Z_n''\oplus Z_{n-1}'' \ar[d]^{\phi_n''=(\epsilon'',s'')}_\simeq \\
	F(P_{n-1}')\oplus F(Q_n') \ar@{ >->}[r] & F(P_{n-1})\oplus F(Q_n) \ar@{->>}[r] & F(P_{n-1}'')\oplus F(Q_n'').
}
\]
Set
\[
	\Phi':=(\phi'_{2*})^{-1}\phi'_{2*+1}, \quad \Phi'':=(\phi''_{2*})^{-1}\phi''_{2*+1}, \quad \tilde{\Phi}:=\tilde{\phi}_{2*}^{-1}\tilde{\phi}_{2*+1}.
\]
Then we obtain a sequence in $\Rel(F)$
\[
	(P_{2*}'\oplus Q_{2*+1}',\Phi',P_{2*-1}'\oplus Q_{2*}') \to (P_{2*}\oplus Q_{2*+1},\tilde{\Phi},P_{2*-1}\oplus Q_{2*}) 
		\to (P_{2*}''\oplus Q_{2*+1}'',\Phi'',P_{2*-1}'\oplus Q_{2*}''),
\]
and it is an exact sequence since the given exact sequences are degree-wise exact.

On the other hand, $\tilde{\Phi}$ is equal to $\Phi:=\phi_{2*}^{-1}\phi_{2*+1}$, $\phi_n:=(\epsilon,s)$, modulo elementary transformations.
Therefore,
\[
\begin{split}
	\chi(P,\alpha,Q) &= [(P_{2*}\oplus Q_{2*+1},\Phi,P_{2*-1}\oplus Q_{2*})] \\
					 &= [(P_{2*}\oplus Q_{2*+1},\tilde{\Phi},P_{2*-1}\oplus Q_{2*})] \\
					 &= [(P_{2*}'\oplus Q_{2*+1}',\Phi',P_{2*-1}'\oplus Q_{2*}')] + [(P_{2*}''\oplus Q_{2*+1}'',\Phi'',P_{2*-1}'\oplus Q_{2*}'')] \\
					 &= \chi(P',\alpha',Q') + \chi(P'',\alpha'',Q'').
\end{split}
\]
\end{proof}

\begin{lemma}\label{lem:chi-H}
Let $P,Q$ be bounded complexes in $\mcal{A}$ with a homotopy equivalence $\alpha\colon F(P)\xrightarrow{\sim}F(Q)$ such that $H_*X,H_*Y\in\mcal{A}$.
Then we have
\[
	\chi(P,\alpha,Q) = \sum_i(-1)^i [(H_iP,H_i\alpha,H_iQ)].
\]
\end{lemma}
\begin{proof}
Since $\chi(P[1],\alpha[1],Q[1])]=-\chi(P,\alpha,Q)$ by Lemma \ref{lem:chi-first} (ii), we may assume that $P_i=Q_i=0$ for $i<0$.
It is easy to see from our assumptions that the kernels and the images of $d_P\colon P_n\to P_{n-1}$ and $d_Q\colon Q_n\to Q_{n-1}$ are in $\mcal{A}$.
Now, we have an exact sequence
\[
\xymatrix@1{
	(\tau_{\ge n+1}P,\tau_{\ge n+1}\alpha,\tau_{\ge n+1}Q) \ar@{ >->}[r] & (P,\alpha,Q) \ar@{->>}[r] & (\tau_{\le n}P,\tau_{\le n}\alpha,\tau_{\le n}Q).
}
\]
By Lemma \ref{lem:chi-exact} and by induction, we may assume that $P_i=Q_i=0$ for $i\ge 2$ and that $d_P\colon P_1\to P_0$ and $d_Q\colon Q_1\to Q_0$ are admissible monomorphisms.
Now, the cone of $\alpha$ has the form
\[
	F(P_1) \to F(P_0)\oplus F(Q_1) \to F(Q_0),
\]
which fits into a commutative diagram
\[
\xymatrix{
									& F(P_1) \ar@{=}[r] \ar@{ >->}[d]^\epsilon 			   & F(P_1) \ar@{ >->}[d]^d \\
	F(Q_1) \ar@{ >->}[r] \ar@{=}[d] & F(P_0)\oplus F(Q_1) \ar@{->>}[d]^\delta \ar@{->>}[r] & F(P_0) \ar@{->>}[d]^{\bar{\alpha}} \\
	F(Q_1) \ar@{ >->}[r] 			& F(Q_0) \ar@{->>}[r] 			  					   & F(H_0(Q))
}
\]
with exact rows and columns.
We take a split $s''$ of $\bar{\alpha}$, and take $\tilde{s}\colon F(Q_0)\to F(P_0)\oplus F(Q_1)$ as in Lemma \ref{lem:halg}.
Then we have a commutative diagram
\[
\xymatrix{
	F(Q_1) \ar@{ >->}[r] \ar@{=}[d] & F(P_0)\oplus F(Q_1) \ar[d]^{\tilde{\phi}:=(\epsilon,\tilde{s})}_\simeq \ar@{->>}[r] & F(P_0) \ar[d]^{\psi:=(d,s'')}_\simeq \\
	F(Q_1) \ar@{ >->}[r] 			& F(Q_0)\oplus F(P_1) \ar@{->>}[r] 		 											  & F(H_0(Q))\oplus F(P_1).
}
\]
Since the rows lift canonically to exact sequences in $\mcal{A}$, we have
\[
	\chi(P,\alpha,Q)=[(P_0\oplus Q_1,\tilde{\phi},P_1\oplus Q_0)]=[(P_0,\psi,H_0(Q)\oplus P_1)].
\]
Finally, it follows from the exact sequence
\[
\xymatrix@1{
	(P_1,1,P_1) \ar@{ >->}[r] & (P_0,\psi,H_0(Q)\oplus P_1) \ar@{->>}[r] & (H_0(P),H_0(\alpha),H_0(Q))
}
\]
that
\[
	[(P_0,\psi,H_0(Q)\oplus P_1)]=[(H_0(P),H_0(\alpha),H_0(Q))].
\]
\end{proof}

\begin{corollary}\label{cor:lemmas}
\leavevmode
\begin{enumerate}[label={\upshape(\roman*)}]
\item Let $f\colon P\to P'$ and $g\colon Q\to Q'$ be quasi-isomorphisms of bounded complexes in $\mcal{A}$
with homotopy equivalences $\alpha\colon F(P)\xrightarrow{\sim}F(Q)$ and $\alpha'\colon F(P')\xrightarrow{\sim}F(Q')$ such that $\alpha'F(f)=F(g)\alpha$.
Then
\[
	\chi(P,\alpha,Q)=\chi(P',\alpha',Q').
\]
\item Let $P,Q$ be bounded complexes in $\mcal{A}$ with a homotopy equivalence $\alpha\colon F(P)\xrightarrow{\sim} F(Q)$.
Suppose that $\alpha$ is homotopic to another homotopy equivalence $\beta\colon F(P)\xrightarrow{\sim}F(Q)$.
Then
\[
	\chi(P,\alpha,Q)=\chi(P,\beta,Q).
\]
\end{enumerate}
\end{corollary}
\begin{proof}
Let $C(F(P))$ be the mapping cylinder of the identity map of $F(P)$.
Then $\alpha$ and $\beta$ extend to a homotopy equivalence $C(F(P))\to F(Q)$.
Since the canonical map $F(P)\to C(F(P))$ lifts to a quasi-isomorphism in $\mcal{A}$, (ii) follows from (i).

For (i), we have an exact sequence
\[
\xymatrix@1{
	(P',\alpha',Q') \ar@{ >->}[r] & (\cone f,\gamma,\cone g) \ar@{->>}[r] & (P[1],\alpha[1],Q[1]).
}
\]
According to Lemma \ref{lem:chi-exact}, we have
\[
\begin{split}
	\chi(\cone f,\gamma,\cone g) &= \chi(P',\alpha',Q')+\chi(P[1],\alpha[1],Q[1]) \\
								 &= \chi(P',\alpha',Q')-\chi(P,\alpha,Q).
\end{split}
\]
Since $H_*\cone f=H_*\cone g= 0$, it follows from Lemma \ref{lem:chi-H} that
\[
	\chi(\cone f,\gamma,\cone g) = \sum(-1)^i[(H_i\cone f,H_i\gamma,H_i\cone g)] = 0.
\]
\end{proof}

\begin{lemma}\label{lem:chi-composite}
Let $P,Q,R$ be bounded complexes in $\mcal{A}$ with homotopy equivalences $\alpha\colon F(P)\xrightarrow{\sim}F(Q)$ and $\beta\colon F(Q)\xrightarrow{\sim}F(R)$.
Then
\[
	\chi(P,\beta\alpha,R) = \chi(P,\alpha,Q) + \chi(Q,\beta,R).
\]
\end{lemma}
\begin{proof}
From the exact sequence
\[
\xymatrix@1{
	(P,\beta\alpha,R) \ar@{ >->}[r] & (P\oplus Q,\left({\begin{smallmatrix}0 & -1 \\ \beta\alpha & 0\end{smallmatrix}}\right),Q\oplus R) \ar@{->>}[r] & (Q,1,Q),
}
\]
it follows that
\[
	\chi(P,\beta\alpha,R) = \chi(P\oplus Q,\left(\begin{smallmatrix}0 & -1 \\ \beta\alpha & 0\end{smallmatrix}\right),Q\oplus R).
\]
Let $\beta^{-1}$ be a homotopy inverse of $\beta$.
Then we have homotopy equivalences
\[
	\begin{pmatrix} \alpha & 0 \\ 0 & \beta \end{pmatrix}
	\sim \begin{pmatrix} 0 & \beta^{-1} \\ \beta & 0 \end{pmatrix}\begin{pmatrix} 0 & -1 \\ \beta\alpha & 0 \end{pmatrix}
	\sim \gamma\begin{pmatrix} 0 & -1 \\ \beta\alpha & 0 \end{pmatrix},
\]
where $\gamma$ is a product of elementary transformations.
Hence, by Corollary \ref{cor:lemmas} (ii) and Lemma \ref{lem:chi-first} (i), we have
\[
\begin{split}
	\chi(P\oplus Q,\left({\begin{smallmatrix}0 & -1 \\ \beta\alpha & 0\end{smallmatrix}}\right),Q\oplus R)
	&=\chi(P\oplus Q,\alpha\oplus\beta,Q\oplus R) \\
	&=\chi(P,\alpha,Q) + \chi(Q,\beta,R).
\end{split}
\]
\end{proof}

\begin{proposition}
The Euler characteristic defined in Definition \ref{def:chi} gives a group homomorphism
\[
	\chi\colon K_0(\msf{D}(F))\to K_0(F).
\]
\end{proposition}
\begin{proof}
Let $X=(P,\bar{\alpha},Q)$ be an object of $\Rel(\msf{K}^b(\mcal{A})/\msf{K}^{b,\emptyset}(\mcal{A})\xrightarrow{\msf{D}(F)}\msf{K}^b(\mcal{B}))$;
$P,Q$ are bounded complexes in $\mcal{A}$ and $\bar{\alpha}$ is the homotopy equivalent class of a homotopy equivalence $\alpha\colon F(P)\xrightarrow{\sim}F(Q)$.
Hence, by Corollary \ref{cor:lemmas} (ii), the Euler characteristic of $X$
\[
	\chi(X):=\chi(P,\alpha,Q)
\]
is well-defined.
It remains to show that $\chi$ kills the relations (a) (b) of $K_0(\msf{D}(F))$ in Definition \ref{def:K_0(F)tri}.

For the relation (b), let $(P,\bar{\alpha},Q),(Q,\bar{\beta},R)\in \Rel(\msf{D}(F))$.
Then $\bar{\beta}\bar{\alpha}$ is a homotopy equivalent class of $\beta\alpha$.
Hence, by Lemma \ref{lem:chi-composite}, we have
\[
	\chi(P,\bar{\beta}\bar{\alpha},R) = \chi(P,\bar{\alpha},Q) + \chi(Q,\bar{\beta},R).
\]

For the relation (a), let
\[
\xymatrix@1{
	(P_1,\alpha_1,Q_1) \ar[r]^-{(f,g)} & (P_2,\alpha_2,Q_2) \ar[r] & (P_3,\alpha_3,Q_3) \ar[r] & (P_1,\alpha_1,Q_1)[1]
}
\]
be an exact triangle in $\Rel(\msf{D}(F))$.
According to Corollary \ref{cor:lemmas} (i), we may assume that $f,g$ are maps of complexes.

Now, there are isomorphism $\beta\colon P_3\xrightarrow{\simeq}\cone f$ and $\gamma\colon Q_3\xrightarrow{\simeq}\cone g$
in $\msf{K}^b(\mcal{A})/\msf{K}^{b,\emptyset}(\mcal{A})$ which make the diagrams
\[
\xymatrix{
	P_2 \ar@{=}[d] \ar[r] & P_3 \ar[d]^\beta \ar[r] & P_1[1] \ar@{=}[d] \\
	P_2 \ar[r] 			  & \cone f \ar[r] 			& P_1[1],
}\quad
\xymatrix{
	Q_2 \ar@{=}[d] \ar[r] & Q_3 \ar[d]^\beta \ar[r] & Q_1[1] \ar@{=}[d] \\
	Q_2 \ar[r] 			  & \cone g \ar[r] 			& Q_1[1],
}
\]
commutative.
We have a homotopy equivalence $\alpha_3'\colon F(\cone f)\xrightarrow{\sim}F(\cone g)$ fitting into the commutative diagram
\[
\xymatrix@!0{
	{} & F(Q_1) \ar[rrrr] & {} & {} & {} & F(Q_2) \ar[rrrr] & {} & {} & {} & F(Q_3) \ar[rrrr] \ar[ldd]_{F(\gamma)} & {} & {} & {} & F(Q_1[1]) \\
	{} & {} & {} & {} & {} & {} & {} & {} & {} & {} & {} & {} & {} & {} \\
	F(Q_1) \ar[rrrr] \ar@{=}[ruu] & {} & {} & {} & F(Q_2) \ar[rrrr] \ar@{=}[ruu] & {} & {} & {} & F(\cone g) \ar[rrrr] & {} & {} & {} & F(Q_1[1]) \ar@{=}[ruu] & {} \\
	{} & F(P_1) \ar'[rrr][rrrr] \ar'[u][uuu]_{\alpha_1} & {} & {} & {} & F(P_2) \ar'[rrr][rrrr] \ar'[u][uuu]_{\alpha_2} 
		& {} & {} & {} & F(P_3) \ar'[rrr][rrrr] \ar'[u][uuu]_{\alpha_3} \ar[ldd]^{F(\beta)} & {} & {} & {} & F(P_1[1]) \ar[uuu]_{\alpha_1[1]} \\
	{} & {} & {} & {} & {} & {} & {} & {} & {} & {} & {} & {} & {} & {} \\
	F(P_1) \ar[rrrr] \ar@{=}[ruu] \ar[uuu]^{\alpha_1} & {} & {} & {} & F(P_2) \ar[rrrr] \ar@{=}[ruu] \ar[uuu]^{\alpha_2}
		& {} & {} & {} & F(\cone f) \ar[rrrr] \ar@{.>}[uuu]^{\alpha_3'} & {} & {} & {} & F(P_1[1]) \ar@{=}[ruu] \ar[uuu]^{\alpha_1[1]} & {}.
}
\]
By Corollary \ref{cor:lemmas} (i), we have
\[
	\chi(P_3,\alpha_3,Q_3) = \chi(\cone f,\alpha_3',\cone g).
\]
By Lemma \ref{lem:chi-exact}, we conclude that
\[
\begin{split}
	\chi(\cone f,\alpha_3',\cone g) &= \chi(P_2,\alpha_2,Q_2) + \chi(P_1[1],\alpha_1[1],Q_1[1]) \\
									&= \chi(P_2,\alpha_2,Q_2) - \chi(P_1,\alpha_1,Q_1).
\end{split}
\]
\end{proof}

\subsection{Proof of Theorem \ref{thm:comparison}}\label{proof:comparison}

We prove that the Euler characteristic $\chi\colon K_0(\msf{D}(F))\to K_0(F)$ is the inverse of the canonical map $\iota\colon K_0(F)\to K_0(\msf{D}(F))$.
It is clear that, for $X\in\Rel(F)$, $\chi\iota[X]=[X]$.
Hence, it remains to prove the following.
\begin{lemma}
For $(P,\alpha,Q)\in\Rel(\msf{D}(F))$,
\[
	[(P,\alpha,Q)] = \iota\chi(P,\alpha,Q)
\]
in $K_0(\msf{D}(F))$.
\end{lemma}
\begin{proof}
We may assume that $P_i=Q_i=0$ for $i<0$.
We prove the lemma by induction on $N:=\min\{n\,|\, P_i=Q_i=0\quad\forall i>n\ge 0\}$.
The case $N=0$ is clear.

We use the following notation:
Set
\[
	\Omega_1:=\bigoplus_i(P_{2i}\oplus Q_{2i+1}) \quad \text{and} \quad \Omega_2:=\bigoplus_i(Q_{2i}\oplus P_{2i+1}),
\]
so that $\chi(P,\alpha,Q)=[(\Omega_1,\Phi,\Omega_2)]$.
According to Freyd-Mitchell theorem, $\mcal{A}$ has an embedding into an abelian category of modules, which allows us talking about elements of objects in $\mcal{A}$.
In principle, we shall denote elements of $P_i$ (resp.\ $Q_i$) by $x_i$ (resp.\ $y_i$).

First of all, we construct $(P',\alpha',Q')\in\Rel(\msf{D}(F))$ with morphisms
\[
	(\Omega_1,\Phi,\Omega_2) \xrightarrow{\theta} (P',\alpha',Q') \xleftarrow{\simeq} (P,\alpha,Q).
\]
Here, $Q':=Q$ and
\[
	P':=P\oplus[\xymatrix@1{\cdots \ar[r] & 0 \ar[r] & Q_1 \ar[r]^1 & Q_1}].
\]
The quasi-isomorphism $\alpha'\colon F(P')\to F(Q)$ is given by
\[
\xymatrix{
	\cdots \ar[r] & F(P_2) \ar[d]^{\alpha_2} \ar[r] & F(P_1)\oplus F(Q_1) \ar[d]^{\alpha_1\oplus 1} \ar[r] & F(P_0)\oplus F(Q_1) \ar[d]^{\alpha_0\oplus d} \\
	\cdots \ar[r] & F(Q_2) \ar[r] 					& F(Q_1) \ar[r] 									   & F(Q_0).
}
\]
The canonical inclusion $P\to P'$ is a quasi-isomorphism, and it yields an isomorphism $(P,\alpha,Q)\xrightarrow{\simeq}(P',\alpha',Q)$ in $\Rel(\msf{D}(F))$.
The map $\theta$ is given by
\begin{align*}
	\Omega_1=P_0\oplus Q_1\oplus P_2\oplus\dotsb &\to P_0'=P_0\oplus Q_1 &(x_0,y_1,x_2,\dotsc)&\mapsto (-x_0,y_1) \\
	\Omega_2=Q_0\oplus P_1\oplus Q_2\oplus\dotsb &\to Q_0 &(y_0,x_1,y_2,\dotsc)&\mapsto y_0.
\end{align*}

We show that $[\cone\theta]=0$ in $K_0(\msf{D}(F))$, which proves the lemma.
First few degrees of $\cone\theta$ look like
\[
\xymatrix@R-1pc{
	\vdots \ar[d] & \vdots \ar[d] \\
	P_2 \ar[d] & Q_2 \ar[d] \\
	(P_0 \oplus Q_1 \oplus P_2 \oplus \dotsb)\oplus P_1\oplus Q_1 \ar[d] & (Q_0 \oplus P_1 \oplus Q_2 \oplus \dotsb)\oplus Q_1 \ar[d] \\
	P_0\oplus Q_1 & Q_0.
}
\]
It follows that the class of $\cone\theta$ in $K_0(\msf{D}(F))$ is equal to $-[(R,\beta,S)]$ where
\begin{gather*}
	R =[\cdots \to P_3 \xrightarrow{d_P} P_2 \xrightarrow{0\oplus d_P} (Q_1\oplus P_2\oplus Q_3\oplus \dotsb)\oplus P_1] \\
	S =[\cdots \to Q_3 \xrightarrow{d_Q} Q_2 \xrightarrow{0\oplus d_Q} (P_1\oplus Q_2\oplus P_3\oplus \dotsb)\oplus Q_1],
\end{gather*}
$\beta_i=\alpha_{i+1}$ for $i\ge 1$ and $\beta_1$ is given by
\[
	\beta_1((y_1,x_2,y_3,\dotsc,),x_1) = (\mrm{pr}_2\Phi(-dx_1,y_1,x_2,y_3,\dotsc),\alpha x_1 + y_1).
\]
The induction hypothesis implies that $[(R,\beta,S)] = \iota\chi(R,\beta,S)$.
We show that $\chi(R,\beta,S)=0$.

We write
\[
	\Omega_1':=\bigoplus_{i\ge 1}(Q_{2i-1}\oplus P_{2i}) \quad \text{and} \quad \Omega_2':=\bigoplus_{i\ge 1}(P_{2i-1}\oplus Q_{2i}),
\]
and denote the projections $\Omega_l\to\Omega_l'$ by $\mrm{pr}_l$.
Then we have
\[
	\chi(R,\beta,S) = [(\Omega_1'\oplus\Omega_2',\Phi',\Omega_2'\oplus\Omega_1')],
\]
where $\Phi'$ is given by
\[
	((y_1,x_2,y_3,\dotsc),(x_1,y_2,x_3,\dotsc))
	\mapsto (-\mrm{pr}_2\Phi(-dx_1,y_1,x_2,y_3,\dotsc),-y_1+\mrm{pr}_1\Phi^{-1}(0,x_1,y_2,x_3,\dotsc)).
\]

Observe that we have $\chi(R,\beta,S) \looparrowright (\Omega_1\oplus\Omega_2,\Psi,\Omega_2\oplus\Omega_1)$ (see Definition \ref{def:looparrow} for ``$\looparrowright$'') with
\begin{multline*}
	\Psi\colon ((x_0,y_1,x_2,\dotsc),(y_0,x_1,y_2,\dotsc))\\
	\mapsto ((y_0,-\mrm{pr}_2\Phi(-dx_1+x_0,y_1,x_2,y_3,\dotsc)),(-x_0,-y_1+\mrm{pr}_1\Phi^{-1}(y_0,x_1,y_2,x_3,\dotsc))).
\end{multline*}
Since the class of $(\Omega_2\oplus\Omega_1,-\Phi^{-1}\oplus\Phi,\Omega_2\oplus\Omega_1)$ is zero, we have
\[
	\chi(R,\beta,S) =[(\Omega_2\oplus\Omega_1,\Psi(-\Phi^{-1}\oplus\Phi),\Omega_2\oplus\Omega_1)].
\]
Now, $\Psi(-\Phi^{-1}\oplus\Phi)$ is given by
\begin{multline*}
	((y_0,x_1,y_2,x_3\dotsc),(x_0,y_1,x_2,y_3\dotsc))\\
	\mapsto ((-\alpha x_0+dy_1,x_1+A,y_2+B,x_3,y_4,\dotsc),(s(y_0)_{P_0}-dx_1,y_1-\alpha x_1+dy_2+s(y_0)_{Q_1},x_2,y_3,\dotsc)),
\end{multline*}
where $A:=s((d_P s(x_0,y_1)_{P_1},0))_{P_1}$, $B:=s((d_P s(x_0,y_1)_{P_1},0))_{Q_2}$ and $s$ is the split $Q_0\to P_0\oplus Q_1$ or $P_0\oplus Q_1\to P_1\oplus Q_2$.
Hence,
\begin{multline*}
	 ((Q_0\oplus P_1\oplus Q_2)\oplus(P_0\oplus Q_1),\Psi',(Q_0\oplus P_1\oplus Q_2)\oplus(P_0\oplus Q_1)) \\
	 \looparrowright (\Omega_2\oplus\Omega_1,\Psi(-\Phi^{-1}\oplus\Phi),\Omega_2\oplus\Omega_1),
\end{multline*}
where $\Psi'$ is the restriction of $\Psi(-\Phi^{-1}\oplus\Phi)$.
We calculate the class of the left hand side in $K_0(F)$ and show that it is zero.

We set $M_0:=Q_0$, $M_2:=P_1\oplus Q_2$, $M_1:=P_0\oplus Q_1$, and denote by $\delta\colon M_l\to M_{l-1}$ the differential of the cone of $\alpha$.
Since $-d_P s(x_0,y_1)_{P_1} + (s\delta(x_0,y_1))_{P_0} = x_0$, we have 
\[
	(A,B)=-s(x_0)+s(s(\delta(x_0,y_1))_{P_0}).
\]
Let $p$ and $q$ be the projections $M_1\to P_0$ and $M_1\to Q_1$ respectively.
Then $\Psi'$ is expressed by the matrix (an endomorphism of $F(M_0)\oplus F(M_2)\oplus F(M_1)$)
\[
	\Psi' = \begin{pmatrix}
				0 & 0 & \delta \\
				0 & 1 & -sp+sps\delta \\
				s & \delta & p_{Q_1}
		   \end{pmatrix}
\]
and we have
\[
	\begin{pmatrix} 1 & 0 & \delta \\ 0 & 1 & 0 \\ 0 & 0 & 1 \end{pmatrix}
	\begin{pmatrix} 1 & 0 & 0 \\ 0 & 1 & 0 \\ -s & 0 & 1 \end{pmatrix}
	\begin{pmatrix} 1 & 0 & 0 \\ -sps & 1 & 0 \\ 0 & -\delta & 1 \end{pmatrix}\Psi'
	= \begin{pmatrix} -1 & 0 & \delta p \\ 0 & 1 & -sp \\ 0 & 0 & 1 \end{pmatrix}.
\]
Therefore, $\Psi'$ lifts to an automorphism of $M_0\oplus M_2\oplus M_1$ modulo elementary transformations, and thus
\[
	[(M_0\oplus M_2\oplus M_1,\Psi',M_0\oplus M_2\oplus M_1)]=0.
\]
\end{proof}

\section{Relative cycle class map}\label{relcycle}

\subsection{Relative $K$-theory of schemes}

For a scheme $X$, we use the following notation:
\begin{enumerate}[(1)]
\item $\msf{Vec}(X)$ is the category of algebraic vector bundles on $X$.
\item $K(X)$ is Quillen's $K$-theory spectrum of the exact category $\msf{Vec}(X)$.
\item $\msf{D}^b(X)$ is the derived category of bounded complexes of $\mcal{O}_X$-modules.
\item $\msf{D}^\perf(X)\subset \msf{D}^b(X)$ is the full subcategory of perfect complexes.
\end{enumerate}
For the most part in this section, we shall assume that a scheme has an ample family of line bundles, cf.\ \cite[2.1.1]{TT90}.
For example, any scheme quasi-projective over an affine scheme has an ample family of line bundles.
Also, any separated regular noetherian scheme has an ample family of line bundles \cite[II 2.2.7.1]{SGA6}.
If a scheme $X$ has an ample family of line bundles, then Quillen's $K$-theory spectrum $K(X)$ behaves well, that is, it is equivalent to the $K$-theory spectrum of the Waldhausen category of perfect complexes of $X$, cf.\ \cite[3.9]{TT90}.

The following theorem is a consequence of the results in \S\ref{relKexact} and \S\ref{relKtri}.
\begin{theorem}\label{thm:relKscheme}
Let $X$ be a scheme with an ample family of line bundles, $Y$ an affine scheme and $f\colon Y\to X$ a morphism of schemes.
Then there exists a natural isomorphism
\[
	\pi_0\hofib(K(X)\xrightarrow{f^*}K(Y)) \simeq K_0(\msf{D}^\perf(X)\xrightarrow{Lf^*}\msf{D}^\perf(Y)).
\]
See Definition \ref{def:K_0(F)tri} for the definition of the right group.
\end{theorem}
\begin{proof}
Since $X$ has an ample family of line bundles, every perfect complex is quasi-isomorphic to a bounded complex of algebraic vector bundles, and thus there is an equivalence of triangulated categories
\[
	\msf{K}^b(\msf{Vec}(X))/\msf{K}^{b,\emptyset}(\msf{Vec}(X)) \xrightarrow{\simeq} \msf{D}^\perf(X).
\]
Since $Y$ is affine, $\msf{Vec}(Y)$ is split exact and $\msf{D}^\perf(Y)\simeq \msf{K}^b(\msf{Vec}(Y))$.
Now, the triangulated functor $Lf^*\colon \msf{D}^\perf(X) \to \msf{D}^\perf(Y)$ is identified with the functor $D(f^*)$ induced from the exact functor $f^*\colon \msf{Vec}(X) \to \msf{Vec}(Y)$, cf.\ \S\ref{comparison}.
Therefore, by Theorem \ref{thm:comparison}, we have an isomorphism
\[
	K_0(\msf{Vec}(X)\xrightarrow{f^*}\msf{Vec}(Y)) \simeq K_0(\msf{D}^\perf(X)\xrightarrow{Lf^*}\msf{D}^\perf(Y)).
\]
By Theorem \ref{thm:Heller}, the left hand side is isomorphic to $\pi_0\hofib(K(X)\xrightarrow{f^*}K(Y))$, and we get the theorem.
\end{proof}

Suppose we are given two schemes $X,Y$ and a morphism of schemes $f\colon Y\to X$ between them.
If the morphism $f$ is obvious from the context, we denote by $K(X,Y)$ the homotopy fiber of $f^*\colon K(X)\to K(Y)$ and write $K_0(X,Y):=\pi_0K(X,Y)$.
We adapted this notation because our main interest is the case $Y$ is a closed subscheme of $X$; in this case, the map $f\colon Y\to X$ is the canonical inclusion.

It is clear from the definition that $K_0(X,Y)$ is contravariant functorial, i.e.\ a commutative diagram
\[
\xymatrix{
	Y'\ar[r] \ar[d] & X' \ar[d] \\
	Y \ar[r] 		& X
}
\]
induces a group homomorphism $K_0(X',Y')\to K_0(X,Y)$.

According to the base change theorem \cite[IV 3.1.1]{SGA6}\footnotemark, we have a proper transfer of $K_0(X,Y)$ in the following case.
\footnotetext{``La conjecture de finitude [loc.\ cit., III 2.1]'' assumed there has been proved by Kiehl in \cite{Ki72}.}

\begin{proposition}\label{prop:transfer}
Suppose we are given a cartesian diagram
\[
\xymatrix{
	Y'\ar[r]^{f'} \ar[d]^{g'} & X' \ar[d]^{g} \\
	Y \ar[r]^{f} 			  & X
}
\]
of schemes.
Assume that:
\begin{enumerate}[label={\upshape(\roman*)}]
\item $X,X'$ have ample family of line bundles and $Y,Y'$ are affine.
\item $f$ and $g$ are Tor-independent over $X$ [loc.\ cit., III 1.5].
\item $g$ is proper and perfect [loc.\ cit., III 4.1].
\end{enumerate}
Then there is a map
\[
	K_0(X',Y') \to K_0(X,Y)
\] 
which sends $(P,\alpha,Q)\in\Rel(Lf^{\prime*})$ to $(Rg_*P,Rg'_*\alpha,Rg_*Q)\in\Rel(Lf^*)$.
\end{proposition}

\subsection{Coniveau filtration}\label{coniveau}

Let $X$ be a scheme of dimension $d$ which has an ample family of line bundles, $Y$ an affine closed subscheme of $X$, and we denote the inclusion $Y\hookrightarrow X$ by $\iota$.
We assume that $X\setminus Y$ is regular.
This is a standing assumption that prevails in \S\ref{coniveau}.

By Theorem \ref{thm:relKscheme}, we identify $K_0(X,Y)$ with the $K_0$ of the triangulated functor $Lf^*\colon \msf{D}^\perf(X)\to \msf{D}^\perf(Y)$.
\begin{definition}\label{def:coniveau}
\leavevmode
\begin{enumerate}[(i)]
\item For $\mfk{A}=(P,\alpha,Q)\in\Rel(L\iota^*)$, let $S_{\mfk{A}}$ be the set of open neighborhoods $U$ of $Y$ in $X$ such that there exits an isomorphism $\tilde{\alpha}\colon P\vert_U\xrightarrow{\simeq}Q\vert_U$ in $\msf{D}^\perf(U)$ which lifts $\alpha$.
\item For $-1\le i\le d$, we define $F_iK_0(X,Y)$ to be the subgroup of $K_0(X,Y)$ generated by elements $\mfk{A}\in\Rel(L\iota^*)$ for which there exists $U\in S_{\mfk{A}}$ with $\dim (X\setminus U)\le i$.
\end{enumerate}
\end{definition}

By the definition, $F_iK_0(X,Y)\subset F_{i+1}K_0(X,Y)$, and $F_{-1}K_0(X,Y)$ is generated by $(P,\alpha,Q)$ for which there exists $\tilde{\alpha}\colon P\xrightarrow{\simeq}Q$ such that $L\iota^*\tilde{\alpha}=\alpha$.
Hence, it follows from Lemma \ref{lem:K_0(F)tri} that $F_{-1}K_0(X,Y)=0$.
In general, $F_dK_0(X,Y)$ may not be equal to $K_0(X,Y)$.
However, we have:
\begin{lemma}\label{lem:coniveau}
If $Y$ has an affine open neighborhood in $X$, then $F_dK_0(X,Y)=K_0(X,Y)$.
\end{lemma}
\begin{proof}
Let $\mfk{A}\in\Rel(L\iota^*)$.
According to Theorem \ref{thm:comparison}, the class of $\mfk{A}$ in $K_0(X,Y)$ is equal to the one of some $(P,\alpha,Q)$ with $P,Q\in\msf{Vec}(X)$.
It suffices to show that $S_{(P,\alpha,Q)}\ne\emptyset$, i.e.\ there exists an open neighborhood $U$ of $Y$ in $X$ and an isomorphism $P\vert_U\xrightarrow{\simeq} Q\vert_U$ which lifts $\alpha$.

By our assumption, we may assume that $X$ is affine, say $X=\Spec A$ and $Y=\Spec A/I$.
Since $Q$ is a projective $A$-module, we have an $A$-homomorphism $\gamma\colon P\to Q$ which fits into the commutative diagram
\[
\xymatrix{
	P \ar@{->>}[d] \ar[r]^\gamma & Q \ar@{->>}[d] \\
	P\otimes_AA/I \ar[r]^\alpha  & Q\otimes_AA/I.
}
\]
Let $K$ and $L$ be the kernel and the cokernel of $\gamma$ respectively.
We claim that for every $y\in Y$, $K_y=L_y=0$.
Since $-\otimes_{A_y}A_y/\mfk{m}_y$ ($\mfk{m}_y$ is the maximal ideal) is right exact and $\gamma_y\otimes_{A_y}A_y/\mfk{m}_y$ is an isomorphism, we have $L_y\otimes_{A_y}A_y/\mfk{m}_y=0$.
By Nakayama's lemma, $L_y=0$.
Since $Q_y$ is projective, the exact sequence
\[
\xymatrix@1{
	0 \ar[r] & K_y \ar[r] & P_y \ar[r] & Q_y \ar[r] & 0
}
\]
is split exact.
Hence, $K_y\otimes_{A_y}A_y/\mfk{m}_y=0$ and $K_y=0$.

Since the supports of $K$ and $L$ are closed, it follows from the claim that there exists some open neighborhood of $Y$ on which $\gamma$ is an isomorphism. 
This completes the proof.
\end{proof}

Let $C_k(X|Y)$ be the set of integral closed subschemes of $X$ of dimension $k$ which do not meet $Y$, and $Z_k(X|Y)$ the free abelian group generated by $C_k(X|Y)$.
For $V\in C_k(X|Y)$, the triple $(\mcal{O}_V,0,0)$ defines an element of $F_kK_0(X,Y)$, which we denote by $\cyc(V)$.
\begin{lemma}\label{lem:cycleclass}
Under the standing assumption in \S\ref{coniveau}, the map
\[
	\cyc\colon Z_k(X|Y) \to F_kK_0(X,Y)/F_{k-1}K_0(X,Y)
\]
is surjective for all $k\ge 0$.
\end{lemma}
\begin{proof}
Suppose we are given $\mfk{A}\in\Rel(L\iota^*)$ whose class is in $F_kK_0(X,Y)$ but not in $F_{k-1}$, so that there exists $U\in S_{\mfk{A}}$ such that $\dim (X\setminus U)= k$.
Set $V:=X\setminus U$ equipped with the reduced scheme structure.

Let $K^V(X)$ be the $K$-theory spectrum of $X$ with support in $V$, i.e.\ the homotopy fiber of the canonical map $K(X)\to K(U)$.
Since $V$ does not meet $Y$, $K^V(X)$ is equivalent to the homotopy fiber of $K(X,Y)\to K(U,Y)$.
Also, since $V$ does not meet $Y$ and $X\setminus Y$ is regular, $K^V(X)$ is equivalent to the $G$-theory spectrum $G(V)$ of $V$.
Hence, we have an exact sequence
\[
\xymatrix@1{
	G_0(V) \ar[r] & K_0(X,Y) \ar[r] & K_0(U,Y).
}
\]
Now, the class of $\mfk{A}$ dies in $K_0(U,Y)$, and thus it comes from $G_0(V)$.

Since the usual cycle map
\[
	\bigoplus_{i=0}^k Z_i(V) \to G_0(V), \quad W\mapsto \mcal{O}_W
\]
is surjective, the class of $\mfk{A}$ is in the image of 
\[
	\bigoplus_{i=0}^k Z_i(V) \to \bigoplus_{i=0}^k Z_i(X|Y) \to F_kK_0(X,Y).
\]
This proves the lemma.
\end{proof}

\subsection{Chow group with modulus}\label{Chowmodulus}

Let $X$ be a scheme separated of finite type over a field $k$ and $D$ an effective Cartier divisor on $X$.
We denote the inclusion $D\hookrightarrow X$ by $\iota$.
We recall the definition of the Chow group with modulus by Binda-Saito \cite{BS17}.

Let $k\ge 0$.
Let $R_k(X|D)$ be the set of integral closed subschemes $V$ of $X\times\mbb{P}^1$ of dimension $k+1$ which are dominant over $\mbb{P}^1$ and satisfy the following condition (modulus condition): 
Let $V^N$ be the normalization of $V$ and $\phi$ the canonical map $V^N\to X\times\mbb{P}^1$, then we have an inequality of Cartier divisors
\[
	\phi^*(D\times\mbb{P}^1) \le \phi^*(X\times\{\infty\}).
\]
For each $V\in R_k(X|D)$, the inverse images $V_t$ of $t\in\mbb{P}^1$ are purely of dimension $k$ and do not meet $D$,
and thus define elements of $Z_k(X|D)$ in the standard way.
\begin{definition}[Binda-Saito \cite{BS17}]
The \textit{Chow group with modulus $\CH_k(X|D)$} is defined to be the quotient of $Z_k(X|D)$ by the relations $[V_0]=[V_1]$ for all $V\in R_k(X|D)$.
\end{definition}

\begin{theorem}\label{thm:cycleclass}
Suppose that $X$ is regular and that $D$ is affine.
Then there is a surjective group homomorphism
\[
	\cyc\colon \CH_k(X|D) \twoheadrightarrow F_kK_0(X,D)/F_{k-1}K_0(X,D)
\]
for every $k\ge 0$.
\end{theorem}
\begin{proof}
Since $X$ is separated regular noetherian, it has an ample family of line bundles \cite[II 2.2.7.1]{SGA6}.
Hence, the assumption in \S\ref{coniveau} is satisfied, and by Lemma \ref{lem:cycleclass}, we have a surjective homomorphism
\[
	\cyc\colon Z_k(X|D) \twoheadrightarrow F_kK_0(X,D)/F_{k-1}K_0(X,D).
\]
We show this map factors through $\CH_k(X|D)$.

We have to show that, for all $V\in R_k(X|D)$, $\cyc(V_0)=\cyc(V_1)$ in $F_kK_0(X,D)/F_{k-1}K_0(X,D)$. 
Let $V\in R_k(X|D)$ and consider the diagram
\[
\xymatrix{
	V^N \ar[d]^\pi \ar@/^1pc/[rrd]^{q_N} \ar@/_3pc/[rdd]_{p_N} & 				 				& \\
	V \ar@{^{(}->}[r] \ar[rd]_p \ar@/^1pc/[rr]^q 			   & X\times\mbb{P}^1 \ar[d] \ar[r] & X \\
															   & \mbb{P}^1 					 	&
}
\]
where $\pi$ is the normalization and the other maps are the obvious ones.
We fix a parameter $t$ of $\mbb{P}^1\setminus\{\infty\}$.
Let $j_0\colon \mcal{O}(-1)\to \mcal{O}_{\mbb{P}^1}$ be the canonical inclusion (sending $t$ to $t$) and $j_1\colon \mcal{O}(-1)\to \mcal{O}_{\mbb{P}^1}$ the map sending $t$ to $t-1$.
Now, $V_t$ ($t=0,1$) are the closed subschemes of $V$ defined by $p^*j_t$, and these are regarded as closed subschemes of $X$ by the restriction of $q$.

By Theorem \ref{thm:relKscheme}, we identify $K_0(V^N,q_N^*D)$ with the $K_0$ of the triangulated functor $L\iota^{\prime *}\colon\msf{D}^\perf(V^N)\to\msf{D}^\perf(q_N^*D)$, where $\iota'\colon q_N^*D\hookrightarrow V^N$ is the canonical inclusion.
Let $V_t'$ be the scheme theoretic inverse image of $V_t$ by $\pi\colon V^N\to V$.
Then the triple $(\mcal{O}_{V_t'},0,0)$ gives an element of $\Rel(L\iota^{\prime *})$.

First, we show that
\[
	[(\mcal{O}_{V_0'},0,0)] = [(\mcal{O}_{V_1'},0,0)]
\]
in $K_0(V^N,q_N^*D)$.
Now, we have an exact sequence
\[
\xymatrix@1{
	0 \ar[r] & p_N^*\mcal{O}(-1) \ar[r]^-{p_N^*j_t} & \mcal{O}_{V^N} \ar[r] & \mcal{O}_{V_t'} \ar[r] & 0
}
\]
of $\mcal{O}_{V^N}$-modules.
Hence, the class of $(\mcal{O}_{V_t'},0,0)$ in $K_0(V^N,q^*D)$ is equal to the one of
\[
	(p_N^*\mcal{O}(-1), p_N^*j_t, \mcal{O}_{V^N}).
\]

Let $\theta$ be the multiplication of $(t-1)/t$, which is defined on $\mbb{P}^1\setminus\{0\}$ and an automorphism on $\mbb{P}^1\setminus\{0,1\}$;
$\theta$ fits into the commutative diagram
\[
\xymatrix@R-1pc{
											 & \mcal{O}_{\mbb{P}^1} \ar@{.>}[dd]^\theta \\
	\mcal{O}(-1) \ar[ru]^{j_0} \ar[rd]_{j_1} & \\
											 & \mcal{O}_{\mbb{P}^1}.
}
\]
It follows that
\[
	[(p_N^*\mcal{O}(-1),p_N^*j_0,\mcal{O}_{V^N})] + [(\mcal{O}_{V^N},p_N^*\theta,\mcal{O}_{V^N})]
	=[(p_N^*\mcal{O}(-1),p_N^*j_1,\mcal{O}_{V^N})].
\]
It is clear that the restriction of $p_N^*\theta$ on $q_N^*(\{\infty\})=\phi^*(X\times\{\infty\})$ is the identity.
Hence, by the modulus condition, the restriction of $p_N^*\theta$ on $\phi^*(D\times\mbb{P}^1)=q_N^*D$ is the identity.
This implies that the second term of the above equation is zero; in other words,
\[
	[(p_N^*\mcal{O}(-1),p_N^*j_0,\mcal{O}_{V^N})]=[(p_N^*\mcal{O}(-1),p_N^*j_1,\mcal{O}_{V^N})].
\]
This proves $[(\mcal{O}_{V_0'},0,0)] = [(\mcal{O}_{V_1'},0,0)]$ in $K_0(V^N,q_N^*D)$.

Now, $\iota\colon D\hookrightarrow X$ and $q_N\colon V^N\to X$ are Tor-independent, and $q_N$ is proper and perfect since $X$ is regular.
Hence, by Proposition \ref{prop:transfer}, we have a transfer map
\[
	q_{N*}\colon K_0(V^N,q_N^*D) \to K_0(X,D).
\]
Consequently, we have
\[
	[(q_{N*}\mcal{O}_{V_0'},0,0)] = [(q_{N*}\mcal{O}_{V_1'},0,0)]
\]
in $K_0(X,D)$.

Finally, we claim that
\[
	\cyc(V_t) \equiv [(\mcal{O}_{V_t},0,0)] \equiv [(q_{N*}\mcal{O}_{V_t'},0,0)]
\]
modulo $F_{k-1}K_0(X,D)$, which completes the proof.
The first term is $\sum_i m_i[(\mcal{O}_{V_{t,i}},0,0)]$ by definition, where $V_{t,i}$ are irreducible components of $V_t$ and $m_i$ are their multiplicity.
By Lemma \ref{lem:multiplicity} below, it suffices to compare the length of $\mcal{O}_{V_t}$ and $q_{N*}\mcal{O}_{V_t'}$ at the generic point of $V_{t,i}$.
This is clear because $V_t\hookrightarrow V$ and $V_t'\hookrightarrow V^N$ are defined by the same rational function.
\end{proof}

\begin{lemma}\label{lem:multiplicity}
Let $\mcal{F}$ be a coherent sheaf on $X$ whose support is of dimension $k$ and disjoint from $D$.
Then
\[
	[(\mcal{F},0,0)] = \sum_{\dim V=k}m_V(\mcal{F})[(\mcal{O}_V,0,0)]
\]
in $F_kK_0(X,D)/F_{k-1}K_0(X,D)$.
Here, $V$ runs over all integral closed subschemes of $X$ of dimension $k$ and $m_V(\mcal{F})$ is the length of the stalk of $\mcal{F}$ at the generic point of $V$.
\end{lemma}
\begin{proof}
Let $j\colon Z\hookrightarrow X$ be the scheme theoretic support of $\mcal{F}$, i.e.\ $\mcal{F}=j_*\mcal{G}$ for some coherent module $\mcal{G}$ of $Z$.
The map $G_0(Z)\to K_0(X,D)$ sending $[\mcal{G}]$ to $[(j_*\mcal{G},0,0)]$ is compatible with the coniveau filtration.
Hence, it suffices to show that
\[
	[\mcal{G}] = \sum_{\dim V=k}m_V(\mcal{G})[\mcal{O}_V]
\]
in $F_kG_0(Z)/F_{k-1}G_0(Z)$.
This is easily verified by induction.
\end{proof}

\end{document}